\documentclass{amsart}
\usepackage{amsmath}
\usepackage{amsfonts}

\newtheorem{theorem}{Theorem}
\theoremstyle{plain}

\newtheorem{corollary}{Corollary}

\newtheorem{example}{Example}

\newtheorem{lemma}{Lemma}

\numberwithin{equation}{section}
\input{tcilatex}

\begin{document}
\title[M. Khan and Faisal]{On fuzzy-$\Gamma $-ideals of $\Gamma $%
-Abel-Grassmann's groupoids }
\author{}
\maketitle

\begin{center}
\bigskip $^{\ast }$\textbf{Madad Khan and Faisal}

\bigskip

\textbf{Department of Mathematics}

\textbf{COMSATS Institute of Information Technology}

\textbf{Abbottabad, Pakistan}

\bigskip

$^{\ast }$\textbf{E-mail: madadmath@yahoo.com, yousafzaimath@yahoo.com}
\end{center}

\bigskip

\textbf{Abstract.} In this paper, we have introduced the notion of $\Gamma $%
-fuzzification in $\Gamma $-AG-groupoids which is in fact the generalization
of fuzzy AG-groupoids. We have studied several properties of an
intra-regular $\Gamma $-AG$^{\ast \ast }$-groupoids in terms of fuzzy $%
\Gamma $-left (right, two-sided, quasi, interior, generalized bi-, bi-)
ideals. We have proved that all fuzzy $\Gamma $-ideals coincide in
intra-regular $\Gamma $-AG$^{\ast \ast }$-groupoids. We have also shown that
the set of fuzzy $\Gamma $-two-sided ideals of an intra-regular $\Gamma $-AG$%
^{\ast \ast }$-groupoid forms a semilattice structure.

\textbf{Keywords}. $\Gamma $-AG-groupoid, intra-regular $\Gamma $-AG$^{\ast
\ast }$-groupoid and fuzzy $\Gamma $-ideals.

\bigskip

\section{\protect\LARGE Introduction}

Abel-Grassmann's groupoid (AG-groupoid) is the generalization of semigroup
theory with wide range of usages in theory of flocks \cite{nas}. The
fundamentals of this non-associative algebraic structure was first
discovered by M. A. Kazim and M. Naseeruddin in $1972$ \cite{Kazi}$.$
AG-groupoid is a non-associative algebraic structure mid way between a
groupoid and a commutative semigroup. It is interesting to note that an
AG-groupoid with right identity becomes a commutative monoid \cite{Mus3}.

After the introduction of fuzzy sets by L. A. Zadeh \cite{L.A.Zadeh} in $%
1965 $, there have been a number of generalizations of this fundamental
concept. A. Rosenfeld \cite{A. Rosenfeld} gives the fuzzification of
algebraic structures and give the concept of fuzzy subgroups. The ideal of
fuzzification in semigroup was first introduced by N. Kuroki \cite{N. Kuroki}%
.

The concept of a $\Gamma $-semigroup has been introduced by M. K. Sen \cite%
{gemma1} in $1981$ as follows: A nonempty set $S$ is called a $\Gamma $%
-semigroup if $x\alpha y\in S$ and $(x\alpha y)\beta z=x\alpha (y\beta z)$
for\ all $x,y,z\in S$ and $\alpha ,\beta \in \Gamma .$ A $\Gamma $-semigroup
is the generalization of semigroup.

In this paper we characterize $\Gamma $-AG$^{\ast \ast }$-groupoids by the
properties of their fuzzy $\Gamma $-ideals and generalize some results. A $%
\Gamma $-AG-groupoid is the generalization of AG-groupoid. Let $S$ and $%
\Gamma $ be any nonempty sets. If there exists a mapping $S\times \Gamma
\times S\rightarrow S$ written as $(x,\alpha ,y)$ by $x\alpha y,$ then $S$
is called a $\Gamma $-AG-groupoid if $x\alpha y\in S$ such that the
following $\Gamma $-left invertive law holds for\ all $x,y,z\in S$ and $%
\alpha ,\beta \in \Gamma $%
\begin{equation}
(x\alpha y)\beta z=(z\alpha y)\beta x.  \tag{$1$}
\end{equation}

A $\Gamma $-AG-groupoid also satisfies the $\Gamma $-medial law for\ all $%
w,x,y,z\in S$ and $\alpha ,\beta ,\gamma \in \Gamma $%
\begin{equation}
(w\alpha x)\beta (y\gamma z)=(w\alpha y)\beta (x\gamma z).  \tag{$2$}
\end{equation}

Note that if a $\Gamma $-AG-groupoid contains a left identity, then it
becomes an AG-groupoid with left identity.

A $\Gamma $-AG-groupoid is called a $\Gamma $-AG$^{\ast \ast }$-groupoid if
it satisfies the following law for\ all $x,y,z\in S$ and $\alpha ,\beta \in
\Gamma $%
\begin{equation}
x\alpha (y\beta z)=y\alpha (x\beta z).  \tag{$3$}
\end{equation}

A $\Gamma $-AG$^{\ast \ast }$-groupoid also satisfies the $\Gamma $%
-paramedial law for\ all $w,x,y,z\in S$ and $\alpha ,\beta ,\gamma \in
\Gamma $%
\begin{equation}
(w\alpha x)\beta (y\gamma z)=(z\alpha y)\beta (x\gamma w).  \tag{$4$}
\end{equation}

\section{Preliminaries}

The following definitions are available in \cite{in}.

Let $S$ be a $\Gamma $-AG-groupoid, a non-empty subset $A$ of $S$ is called
a $\Gamma $-AG-subgroupoid if $a\gamma b\in A$ for all $a$, $b\in A$ and $%
\gamma \in \Gamma $ or if$\ A\Gamma A\subseteq A.$

A subset $A$ of a $\Gamma $-AG-groupoid $S$ is called a $\Gamma $-left
(right) ideal of $S$ if $S\Gamma A\subseteq A$ $\left( A\Gamma S\subseteq
A\right) $ and $A$ is called a $\Gamma $-two-sided-ideal of $S$ if it is
both a $\Gamma $-left ideal and a $\Gamma $-right ideal.

A subset $A$ of a $\Gamma $-AG-groupoid $S$ is called a $\Gamma $%
-generalized bi-ideal of $S$ if $\left( A\Gamma S\right) \Gamma A\subseteq
A. $

A sub $\Gamma $-AG-groupoid $A$ of a $\Gamma $-AG-groupoid $S$ is called a $%
\Gamma $-bi-ideal of $S$ if $\left( A\Gamma S\right) \Gamma A\subseteq A$.

A subset $A$ of a $\Gamma $-AG-groupoid $S$ is called a $\Gamma $-interior
ideal of $S$ if $\left( S\Gamma A\right) \Gamma S\subseteq A.$

A subset $A$ of a $\Gamma $-AG-groupoid $S$ is called a $\Gamma $%
-quasi-ideal of $S$ if $S\Gamma A\cap A\Gamma S\subseteq A.$

\bigskip 

A fuzzy subset $f$ of a given set $S$ is described as an arbitrary function $%
f:S\longrightarrow \lbrack 0,1]$, where $[0,1]$ is the usual closed interval
of real numbers.

Now we introduce the following definitions.

Let $f$ and $g$ be any fuzzy subsets of a $\Gamma $-AG-groupoid $S$, then
the $\Gamma $-product $f\circ _{\Gamma }g$ is defined by

\begin{equation*}
\left( f\circ _{\Gamma }g\right) (a)=\left\{ 
\begin{array}{c}
\bigvee\limits_{a=b\alpha c}\left\{ f(b)\wedge g(c)\right\} \text{, if }%
\exists \text{ }b,c\in S\ni \text{ }a=b\alpha c\text{ where }\alpha \in
\Gamma . \\ 
0,\text{ \ \ \ \ \ \ \ \ \ \ \ \ \ \ \ \ \ \ \ \ \ \ \ \ \ \ \ \ \ \ \ \ \ \
\ \ \ \ \ \ \ \ \ \ \ \ \ \ \ \ \ \ otherwise.}%
\end{array}%
\right.
\end{equation*}

A fuzzy subset $f$ of a $\Gamma $-AG-groupoid $S$ is called a fuzzy $\Gamma $%
-AG-subgroupoid if $f(x\alpha y)\geq f(x)\wedge f(y)$ for all $x$, $y\in S$
and $\alpha \in \Gamma .$

A fuzzy subset $f$ of a $\Gamma $-AG-groupoid $S$ is called fuzzy $\Gamma $%
-left ideal of $S$ if $f(x\alpha y)\geq f(y)$ for all $x$, $y\in S$ and $%
\alpha \in \Gamma .$

A fuzzy subset $f$ of a $\Gamma $-AG-groupoid $S$ is called fuzzy $\Gamma $%
-right ideal of $S$ if $f(x\alpha y)\geq f(x)$ for all $x$, $y\in S$ and $%
\alpha \in \Gamma .$

A fuzzy subset $f$ of a $\Gamma $-AG-groupoid $S$ is called fuzzy $\Gamma $%
-two-sided ideal of $S$ if it is both a fuzzy $\Gamma $-left ideal and a
fuzzy $\Gamma $-right ideal of $S.$

A fuzzy subset $f$ of a $\Gamma $-AG-groupoid $S$ is called fuzzy $\Gamma $%
-generalized bi-ideal of $S$ if $f((x\alpha y)\beta z)\geq f(x)\wedge f(z)$,
for all $x$, $y$ and $z\in S$ and $\alpha ,\beta \in \Gamma .$

A fuzzy $\Gamma $-AG-subgroupoid $f$ of a $\Gamma $-AG-groupoid $S$ is
called fuzzy $\Gamma $-bi-ideal of $S$ if $f((x\alpha y)\beta z)\geq
f(x)\wedge f(z)$, for all $x$, $y$ and $z\in S$ and $\alpha ,\beta \in
\Gamma .$

A fuzzy subset $f$ of a $\Gamma $-AG-groupoid $S$ is called fuzzy $\Gamma $%
-interior ideal of $S$ if $f((x\alpha y)\beta z)\geq f(y)$, for all $x$, $y$
and $z\in S$ and $\alpha ,\beta \in \Gamma .$

A fuzzy subset $f$ of a $\Gamma $-AG-groupoid $S$ is called fuzzy $\Gamma $%
-interior ideal of $S$ if $(f\circ _{\Gamma }S)\cap (S\circ _{\Gamma
}f)\subseteq f$.

\section{$\Gamma $-fuzzification in $\Gamma $-AG-groupoids}

\begin{example}
\label{exp}Let $S=\{1,2,3,4,5,6,7,8,9\}$. The following multiplication table
shows that $S$ is an AG-groupoid and also an AG-band.
\end{example}

\begin{center}
\begin{tabular}{l|lllllllll}
. & $1$ & $2$ & $3$ & $4$ & $5$ & $6$ & $7$ & $8$ & $9$ \\ \hline
$1$ & $1$ & $4$ & $7$ & $3$ & $6$ & $8$ & $2$ & $9$ & $5$ \\ 
$2$ & $9$ & $2$ & $5$ & $7$ & $1$ & $4$ & $8$ & $6$ & $3$ \\ 
$3$ & $6$ & $8$ & $3$ & $5$ & $9$ & $2$ & $4$ & $1$ & $7$ \\ 
$4$ & $5$ & $9$ & $2$ & $4$ & $7$ & $1$ & $6$ & $3$ & $8$ \\ 
$5$ & $3$ & $6$ & $8$ & $2$ & $5$ & $9$ & $1$ & $7$ & $4$ \\ 
$6$ & $7$ & $1$ & $4$ & $8$ & $3$ & $6$ & $9$ & $5$ & $2$ \\ 
$7$ & $8$ & $3$ & $6$ & $9$ & $2$ & $5$ & $7$ & $4$ & $1$ \\ 
$8$ & $2$ & $5$ & $9$ & $1$ & $4$ & $7$ & $3$ & $8$ & $6$ \\ 
$9$ & $4$ & $7$ & $1$ & $6$ & $8$ & $3$ & $5$ & $2$ & $9$%
\end{tabular}
\end{center}

Clearly $S$ is non-commutative and non-associative because $23\neq 32$ and $%
(42)3\neq 4(23).$

Let $\Gamma =\{\alpha ,\beta \}$ and define a mapping $S\times \Gamma \times
S\rightarrow S$ by $a\alpha b=a^{2}b$ and $a\beta b=ab^{2}$ for all $a,b\in
S.$ Then it is easy to see that $S$ is a $\Gamma $-AG-groupoid and also a $%
\Gamma $-AG-band. Note that $S$ is non-commutative and non-associative
because $9\alpha 1\neq 1\alpha 9$ and $(6\alpha 7)\beta 8\neq 6\alpha
(7\beta 8).$

\begin{example}
Let $\Gamma =\{1,2,3\}$ and define a mapping $%
\mathbb{Z}
\times \Gamma \times 
\mathbb{Z}
\rightarrow 
\mathbb{Z}
$ by $a\beta b=b-\beta -a-\beta -z$ for all $a,b,z\in 
\mathbb{Z}
$ and $\beta \in \Gamma ,$ where " $-$ " is a usual subtraction of integers.
Then $%
\mathbb{Z}
$ is a $\Gamma $-AG-groupoid. Indeed%
\begin{eqnarray*}
(a\beta b)\gamma c &=&(b-\beta -a-\beta -z)\gamma c=c-\gamma -(b-\beta
-a-\beta -z)-\gamma -z \\
&=&c-\gamma -b+\beta +a+\beta +z-\gamma -z=c-b+2\beta +a-2\gamma .
\end{eqnarray*}
\end{example}

and%
\begin{eqnarray*}
(c\beta b)\gamma a &=&(b-\beta -c-\beta -z)\gamma a=a-\gamma -(b-\beta
-c-\beta -z)-\gamma -z \\
&=&a-\gamma -b+\beta +c+\beta +z-\gamma -z=a-2\gamma -b+2\beta +c \\
&=&c-b+2\beta +a-2\gamma .
\end{eqnarray*}

Which shows that $(a\beta b)\gamma c=(c\beta b)\gamma a$ for all $a,b,c\in 
\mathbb{Z}
$ and $\beta ,\gamma \in \Gamma .$ This example is the generalizaion of a $%
\Gamma $-AG-groupoid given by T. Shah and I. Rehman $($see \cite{in}$).$

\begin{example}
Assume that $S$ is an AG-groupoid with left identity and let $\Gamma =\{1\}.$
Define a mapping $S\times \Gamma \times S\rightarrow S$ by $x1y=xy$ for all $%
x,y\in S,$ then $S$ is a $\Gamma $-AG-groupoid. Thus we have seen that every
AG-groupoid is a $\Gamma $-AG-groupoid for $\Gamma =\{1\},$ that is, $\Gamma 
$-AG-groupoid is the generalization of AG-groupoid. Also $S$ is a $\Gamma $%
-AG$^{\ast \ast }$-groupoid because $x1(y1z)=y1(x1z)$ for all $x,y,z\in S.$
\end{example}

\begin{example}
Let $S$ be an AG-groupoid and $\Gamma =\{1\}.$ Define a mapping $S\times
\Gamma \times S\rightarrow S$ by $x1y=xy$ for all $x,y\in S,$ then we know
that $S$ is a $\Gamma $-AG-groupoid $.$ Let $L$ be a left ideal of an
AG-groupoid $S$, then $S\Gamma L=SL\subseteq L.$ Thus $L$ is $\Gamma $-left
ideal of $S.$ This shows that every $\Gamma $-left ideal of $\Gamma $%
-AG-groupoid is a generalization of a left ideal in an AG-groupoid $($for
suitable $\Gamma ).$ Similarly all the fuzzy $\Gamma $-ideals are the
generalizations of fuzzy ideals.
\end{example}

By keeping the generalization, the proof of Lemma \ref{sf} and Theorem \ref%
{trm} are same as in \cite{Nouman}.

\begin{lemma}
\label{sf}Let $f$ be a fuzzy subset of a $\Gamma $-AG-groupoid $S,$ then $%
S\circ _{\Gamma }f=f$.
\end{lemma}

\begin{theorem}
\label{trm}Let $S$ be a $\Gamma $-AG-groupoid, then the following properties
hold in $S$.
\end{theorem}

$(i)$ $(f\circ _{\Gamma }g)\circ _{\Gamma }h=(h\circ _{\Gamma }g)\circ
_{\Gamma }f$ for all fuzzy subsets $f$, $g$ and $h$ of $S$.

$(ii)$ $(f\circ _{\Gamma }g)\circ _{\Gamma }(h\circ _{\Gamma }k)=(f\circ
_{\Gamma }h)\circ _{\Gamma }(g\circ _{\Gamma }k)$ for all fuzzy subsets $f$, 
$g$, $h$ and $k$ of $S$.

\begin{theorem}
Let $S$ be a $\Gamma $-AG$^{\ast \ast }$-groupoid, then the following
properties hold in $S$.
\end{theorem}

$(i)$ $f\circ _{\Gamma }(g\circ _{\Gamma }h)=g\circ _{\Gamma }(f\circ
_{\Gamma }h)$ for all fuzzy subsets $f$, $g$ and $h$ of $S$.

$(ii)$ $(f\circ _{\Gamma }g)\circ _{\Gamma }(h\circ _{\Gamma }k)=(k\circ
_{\Gamma }h)\circ _{\Gamma }(g\circ _{\Gamma }f)$ for all fuzzy subsets $f$, 
$g$, $h$ and $k$ of $S$.

\begin{proof}
$(i):$ Assume that $x$ is an arbitrary element of a $\Gamma $-AG$^{\ast \ast
}$-groupoid $S$ and let $\alpha ,\beta \in \Gamma $. If $x$ is not
expressible as a product of two elements in\ $S$, then $\left( f\circ
_{\Gamma }\left( g\circ _{\Gamma }h\right) \right) (x)=0=\left( g\circ
_{\Gamma }\left( f\circ _{\Gamma }h\right) \right) (x).$ Let there exists $y$
and $z$ in $S$ such that $x=y\alpha z,$ then by using $(3)$, we have%
\begin{eqnarray*}
\left( \left( f\circ _{\Gamma }g\right) \circ _{\Gamma }\left( h\circ
_{\Gamma }k\right) \right) (x) &=&\underset{x=y\alpha z}{\bigvee }\left\{
\left( f\circ _{\Gamma }g\right) (y)\wedge \left( h\circ _{\Gamma }k\right)
(z)\right\} \\
&=&\underset{x=y\alpha z}{\bigvee }\left\{ \underset{y=p\beta q}{\bigvee }%
\left\{ f(p)\wedge g(q)\right\} \wedge \underset{z=u\gamma v}{\bigvee }%
\left\{ h(u)\wedge k(v)\right\} \right\} \\
&=&\underset{x=(p\beta q)\alpha (u\gamma v)}{\bigvee }\left\{ f(p)\wedge
g(q)\wedge h(u)\wedge k(v)\right\} \\
&=&\underset{x=(v\beta u)\alpha (q\gamma p)}{\bigvee }\left\{ k(v)\wedge
h(u)\wedge _{\Gamma }g(q)\wedge f(p)\right\} \\
&=&\underset{x=m\alpha n}{\bigvee }\left\{ \underset{m=v\beta u}{\bigvee }%
\left\{ k(v)\wedge h(u)\right\} \wedge \underset{n=q\gamma p}{\bigvee }%
\left\{ g(q)\wedge f(p)\right\} \right\} \\
&=&\underset{x=m\alpha n}{\bigvee }\left\{ (k\circ _{\Gamma }h)(m)\wedge
(g\circ _{\Gamma }f)(n)\right\} \\
&=&\left( (k\circ _{\Gamma }h)\circ _{\Gamma }(g\circ _{\Gamma }f)\right)
(x).
\end{eqnarray*}

If $z$ is not expressible as a product of two elements in $S$, then $\left(
f\circ _{\Gamma }\left( g\circ _{\Gamma }h\right) \right) (x)=0=\left(
g\circ _{\Gamma }\left( f\circ _{\Gamma }h\right) \right) (x).$ Hence, $%
\left( f\circ _{\Gamma }\left( g\circ _{\Gamma }h\right) \right) (x)=\left(
g\circ _{\Gamma }\left( f\circ _{\Gamma }h\right) \right) (x)$ for all $x$
in $S$.

$(ii):$ If any element $x$ of $S$ is not expressible as product of two
elements in $S$ at any stage, then, $\left( \left( f\circ _{\Gamma }g\right)
\circ _{\Gamma }\left( h\circ _{\Gamma }k\right) \right) (x)=0=\left( \left(
k\circ _{\Gamma }h\right) \circ _{\Gamma }\left( g\circ _{\Gamma }f\right)
\right) (x).$ Assume that $\alpha ,\beta ,\gamma \in \Gamma $ and let there
exists $y,z$ in $S$ such that $x=y\alpha z,$ then by using $(4)$, we have%
\begin{eqnarray*}
\left( \left( f\circ _{\Gamma }g\right) \circ _{\Gamma }\left( h\circ
_{\Gamma }k\right) \right) (x) &=&\underset{x=y\alpha z}{\bigvee }\left\{
\left( f\circ _{\Gamma }g\right) (y)\wedge \left( h\circ _{\Gamma }k\right)
(z)\right\} \\
&=&\underset{x=y\alpha z}{\bigvee }\left\{ \underset{y=p\beta q}{\bigvee }%
\left\{ f(p)\wedge g(q)\right\} \wedge \underset{z=u\gamma v}{\bigvee }%
\left\{ h(u)\wedge k(v)\right\} \right\} \\
&=&\underset{x=(p\beta q)\alpha (u\gamma v)}{\bigvee }\left\{ f(p)\wedge
g(q)\wedge h(u)\wedge k(v)\right\} \\
&=&\underset{x=(v\beta u)\alpha (q\gamma p)}{\bigvee }\left\{ k(v)\wedge
h(u)\wedge _{\Gamma }g(q)\wedge f(p)\right\} \\
&=&\underset{x=m\alpha n}{\bigvee }\left\{ \underset{m=v\beta u}{\bigvee }%
\left\{ k(v)\wedge h(u)\right\} \wedge \underset{n=q\gamma p}{\bigvee }%
\left\{ g(q)\wedge f(p)\right\} \right\} \\
&=&\underset{x=m\alpha n}{\bigvee }\left\{ (k\circ _{\Gamma }h)(m)\wedge
(g\circ _{\Gamma }f)(n)\right\} \\
&=&\left( (k\circ _{\Gamma }h)\circ _{\Gamma }(g\circ _{\Gamma }f)\right)
(x).
\end{eqnarray*}
\end{proof}

By keeping the generalization, the proof of the following two lemma's are
same as in \cite{Mordeson}.

\begin{lemma}
\label{fghj}Let $f$ be a \ fuzzy subset of a $\Gamma $-AG-groupoid $S$, then
the following properties hold.
\end{lemma}

$(i)$ $f$ is a fuzzy $\Gamma $-AG-subgroupoid of $S$ if and only if $f\circ
_{\Gamma }f\subseteq f$.

$(ii)$ $f$ is a fuzzy $\Gamma $-left (right) ideal of $S$ if and only if$\
S\circ _{\Gamma }f\subseteq f$ $(f\circ _{\Gamma }S\subseteq f)$.

$(iii)$ $f$ is a fuzzy $\Gamma $-two-sided ideal of $S$ if and only if $%
S\circ _{\Gamma }f\subseteq f$ and $f\circ _{\Gamma }S\subseteq f$.

\begin{lemma}
\label{bi-ideal}Let $f$ be a fuzzy $\Gamma $-AG-subgroupoid of a $\Gamma $%
-AG-groupoid $S$, then $f$ is a fuzzy $\Gamma $-bi-ideal of $S$ if and only
if $(f\circ _{\Gamma }S)\circ _{\Gamma }f\subseteq f$.
\end{lemma}

\begin{lemma}
Let $f$\ be any fuzzy\ $\Gamma $-right ideal and $g$\ be any fuzzy\ $\Gamma $%
-left ideal of $\Gamma $-AG-groupoid $S$, then $f\cap g$\ is a fuzzy$\
\Gamma $-quasi ideal of $S$.
\end{lemma}

\begin{proof}
It is easy to observe the following%
\begin{equation*}
\left( \left( f\cap g\right) \circ _{\Gamma }S\right) \cap \left( S\circ
_{\Gamma }\left( f\cap g\right) \right) \subseteq \left( f\circ _{\Gamma
}S\right) \cap \left( S\circ _{\Gamma }g\right) \subseteq f\cap g\text{.}
\end{equation*}
\end{proof}

\begin{lemma}
\label{qqq}Every fuzzy\ $\Gamma $-quasi ideal of a $\Gamma $-AG-groupoid $S$%
\ is a fuzzy $\Gamma $-AG-subgroupoid of $S$.
\end{lemma}

\begin{proof}
Let $f$ be any fuzzy $\Gamma $-quasi ideal of $S$, then $f\circ _{\Gamma
}f\subseteq f\circ _{\Gamma }S$, and $f\circ _{\Gamma }f\subseteq S\circ
_{\Gamma }f$, therefore 
\begin{equation*}
f\circ _{\Gamma }f\subseteq f\circ _{\Gamma }S\cap S\circ _{\Gamma
}f\subseteq f\text{.}
\end{equation*}%
Hence $f$ is a fuzzy $\Gamma $-AG-subgroupoid of $S$.
\end{proof}

A fuzzy subset $f$ of a $\Gamma $-AG-groupoid $S$ is called $\Gamma $%
-idempotent, if $f\circ _{\Gamma }f=f$.

\begin{lemma}
\label{fully idempotent fqi fbi}In a $\Gamma $-AG-groupoid $S$, every $%
\Gamma $-idempotent fuzzy\ $\Gamma $-quasi ideal is a fuzzy\ $\Gamma $%
-bi-ideal of $S$.
\end{lemma}

\begin{proof}
Let $f$ be any fuzzy $\Gamma $-quasi ideal of $S$, then by lemma \ref{qqq}, $%
f$ is a fuzzy $\Gamma $-AG-subgroupoid. Now by using $(2),$ we have 
\begin{equation*}
(f\circ _{\Gamma }S)\circ _{\Gamma }f\subseteq (S\circ _{\Gamma }S)\circ
_{\Gamma }f\subseteq S\circ _{\Gamma }f\text{ }
\end{equation*}

and%
\begin{eqnarray*}
(f\circ _{\Gamma }S)\circ _{\Gamma }f &=&(f\circ _{\Gamma }S)\circ _{\Gamma
}(f\circ _{\Gamma }f)=(f\circ _{\Gamma }f)\circ _{\Gamma }(S\circ _{\Gamma
}f) \\
&\subseteq &f\circ _{\Gamma }(S\circ _{\Gamma }S)\subseteq f\circ _{\Gamma }S%
\text{. }
\end{eqnarray*}

This implies that $(f\circ _{\Gamma }S)\circ _{\Gamma }f\subseteq (f\circ
_{\Gamma }S)\cap (S\circ _{\Gamma }f)\subseteq f$. Hence by lemma \ref%
{bi-ideal}, $f$ is a fuzzy $\Gamma $-bi-ideal of $S$.
\end{proof}

\begin{lemma}
In a $\Gamma $-AG-groupoid $S$, each one sided fuzzy\ $\Gamma $-$($left,
right$)$ ideal is a fuzzy $\Gamma $-quasi ideal of $S$.
\end{lemma}

\begin{proof}
It is obvious.
\end{proof}

\begin{corollary}
In a $\Gamma $-AG-groupoid $S$, every fuzzy\ $\Gamma $-two- sided ideal of $%
S $\ is a fuzzy\ $\Gamma $-quasi ideal of $S$.
\end{corollary}

\begin{lemma}
In a $\Gamma $-AG-groupoid $S$, each one sided fuzzy\ $\Gamma $-$($left,
right$)$ ideal of $S$ is a fuzzy\ $\Gamma $-generalized bi-ideal of $S$.
\end{lemma}

\begin{proof}
Assume that $f$ be any fuzzy $\Gamma $-left ideal of $S$. Let $a$, $b$, $%
c\in S$ and let $\alpha ,\beta \in \Gamma $. Now by using $(1),$ we have $%
f((a\alpha b)\beta c)\geq f((c\alpha b)\beta a)\geq f(a)$ and $f((a\alpha
b)\beta c)\geq f(c)$. Thus $f((a\alpha b)\beta c)\geq f(a)\wedge f(c)$.
Similarly in the case of fuzzy\ $\Gamma $-right ideal.
\end{proof}

\begin{lemma}
Let $f$ or $g$ is a $\Gamma $-idempotent fuzzy $\Gamma $-quasi ideal of a $%
\Gamma $-AG$^{\ast \ast }$-groupoid $S$,\ then $f\circ _{\Gamma }g$\ or $%
g\circ _{\Gamma }f$ is\ a$\ $fuzzy\ $\Gamma $-bi-ideal of $S$.
\end{lemma}

\begin{proof}
Clearly $f\circ g$ is a fuzzy $\Gamma $-AG-subgroupoid. Now using lemma \ref%
{bi-ideal}, $(1)$, $(4)$ and $(2)$, we have%
\begin{eqnarray*}
((f\circ _{\Gamma }g)\circ _{\Gamma }S)\circ _{\Gamma }(f\circ _{\Gamma }g)
&=&((S\circ _{\Gamma }g)\circ _{\Gamma }f)\circ _{\Gamma }(f\circ _{\Gamma
}g) \\
&\subseteq &((S\circ _{\Gamma }S)\circ _{\Gamma }f)\circ _{\Gamma }(f\circ
_{\Gamma }g) \\
&=&(S\circ _{\Gamma }f)\circ _{\Gamma }(f\circ _{\Gamma }g) \\
&=&(g\circ _{\Gamma }f)\circ _{\Gamma }(f\circ _{\Gamma }S) \\
&=&((f\circ _{\Gamma }S)\circ _{\Gamma }f)\circ _{\Gamma }g\subseteq (f\circ
_{\Gamma }g)\text{.}
\end{eqnarray*}%
Similarly we can show that $g\circ f$ is a fuzzy $\Gamma $-bi-ideal of $S$.
\end{proof}

\begin{lemma}
The product of two fuzzy\ $\Gamma $-left $($right$)$ ideal of a $\Gamma $-AG$%
^{\ast \ast }$-groupoid $S$\ is a fuzzy\ $\Gamma $-left $($right$)$ ideal of 
$S$.
\end{lemma}

\begin{proof}
Let $f$ and $g$ be any two fuzzy $\Gamma $-left ideals of $S$, then by using 
$(3)$, we have 
\begin{equation*}
S\circ _{\Gamma }(f\circ _{\Gamma }g)=f\circ _{\Gamma }(S\circ _{\Gamma
}g)\subseteq f\circ _{\Gamma }g\text{.}
\end{equation*}

Let $f$ and $g$ be any two fuzzy $\Gamma $-right ideals of $S$, then by
using $(2)$, we have 
\begin{equation*}
(f\circ _{\Gamma }g)\circ _{\Gamma }S=(f\circ _{\Gamma }g)\circ _{\Gamma
}(S\circ _{\Gamma }S)=(f\circ _{\Gamma }S)\circ _{\Gamma }(g\circ _{\Gamma
}S)\subseteq f\circ _{\Gamma }g\text{.}
\end{equation*}
\end{proof}

\section{$\Gamma $-fuzzification in intra-regular $\Gamma $-AG$^{\ast \ast }$%
-groupoids}

An element $a$ of a $\Gamma $-AG-groupoid $S$ is called an intra-regular if
there exists $x,$ $y\in S$ and $\beta ,\gamma ,\xi \in \Gamma $ such that $%
a=(x\beta (a\xi a))\gamma y$ and $S$ is called intra-regular if every
element of $S$ is intra-regular.

Note that in an intra-regular $\Gamma $-AG-groupoid $S$, we can write $%
S\circ _{\Gamma }S=S$.

\begin{example}
Let $S=\{1,2,3,4,5\}$ be an AG-groupoid with the following multiplication
table.
\end{example}

\begin{center}
\begin{tabular}{l|lllll}
. & $a$ & $b$ & $c$ & $d$ & $e$ \\ \hline
$a$ & $a$ & $a$ & $a$ & $a$ & $a$ \\ 
$b$ & $a$ & $b$ & $b$ & $b$ & $b$ \\ 
$c$ & $a$ & $b$ & $d$ & $e$ & $c$ \\ 
$d$ & $a$ & $b$ & $c$ & $d$ & $e$ \\ 
$e$ & $a$ & $b$ & $e$ & $c$ & $d$%
\end{tabular}
\end{center}

Let $\Gamma =\{1\}$ and define a mapping $S\times \Gamma \times S\rightarrow
S$ by $x1y=xy$ for all $x,y\in S,$ then $S$ is a $\Gamma $-AG$^{\ast \ast }$%
-groupoid because $(x1y)1z=(z1y)1x$ and $x1(y1z)=y1(x1z)$ for all $x,y,z\in
S.$ Also $S$ is an intra-regular because $a=(b1(a1a))1a,$ $b=(c1(b1b))1d,$ $%
c=(c1(c1c))1d,$ $d=(c1(d1d))1e,$ $e=(c1(e1e))1c.$

It is easy to observe that $\{a,b\}$ is a $\Gamma $-two-sided ideal of an
intra-regular $\Gamma $-AG$^{\ast \ast }$-groupoid $S.$

It is easy to observe that in an intra-regular $\Gamma $-AG-groupoid $S,$
the following holds%
\begin{equation}
S=S\Gamma S\text{.}  \tag{$5$}
\end{equation}

\begin{lemma}
\label{llb}A fuzzy subset $f$ of an intra-regular $\Gamma $-AG-groupoid $S$
is a fuzzy $\Gamma $-right ideal if and only if it is a fuzzy $\Gamma $-left
ideal.
\end{lemma}

\begin{proof}
Assume that $f$ is a fuzzy $\Gamma $-right ideal of $S$. Since $S$ is an
intra-regular $\Gamma $-AG-groupoid, so for each $a\in S$ there exist $%
x,y\in S$ and $\beta ,\xi ,\gamma \in \Gamma $ such that $a=(x\beta (a\xi
a))\gamma y$. Now let $\alpha \in \Gamma ,$ then by using $(1)$, we have 
\begin{equation*}
f(a\alpha b)=f(((x\beta (a\xi a))\gamma y)\alpha b)=f((b\gamma y)\alpha
(x\beta (a\xi a)))\geq f(b\gamma y)\geq f(b).
\end{equation*}

Conversely, assume that $f$ is a fuzzy $\Gamma $-left ideal of $S$. Now by
using $(1)$, we have%
\begin{eqnarray*}
f(a\alpha b) &=&f(((x\beta (a\xi a))\gamma y)\alpha b)=f((b\gamma y)\alpha
(x\beta (a\xi a))) \\
&\geq &f(x\beta (a\xi a))\geq f(a\xi a)\geq f(a)\text{.}
\end{eqnarray*}
\end{proof}

\begin{theorem}
Every fuzzy $\Gamma $-left ideal of an intra-regular $\Gamma $-AG$^{\ast
\ast }$-groupoid $S$ is $\Gamma $-idempotent.
\end{theorem}

\begin{proof}
Assume that $f$ is a fuzzy $\Gamma $-left ideal of $S$, then clearly $f\circ
_{\Gamma }f\subseteq S\circ _{\Gamma }f\subseteq f$. Since $S$ is an
intra-regular $\Gamma $-AG-groupoid, so for each $a\in S$ there exist $%
x,y\in S$ and $\beta ,\gamma ,\xi \in \Gamma $ such that $a=(x\beta (a\xi
a))\gamma y$. Now let $\alpha \in \Gamma $, then by using $(3)$ and $(1),$
we have%
\begin{equation*}
a=(x\beta (a\xi a))\gamma y=(a\beta (x\xi a))\gamma y=(y\beta (x\xi
a))\gamma a\text{.}
\end{equation*}

Thus, we have%
\begin{eqnarray*}
(f\circ _{\Gamma }f)(a) &=&\bigvee_{a=(y\beta (x\xi a))\gamma a}\{f(y\beta
(x\xi a))\wedge f(a)\}\geq f(y\beta (x\xi a))\wedge f(a) \\
&\geq &f(a)\wedge f(a)=f(a)\text{.}
\end{eqnarray*}
\end{proof}

\begin{corollary}
\label{id}Every fuzzy $\Gamma $-two-sided ideal of an intra-regular $\Gamma $%
-AG$^{\ast \ast }$-groupoid $S$ is $\Gamma $-idempotent.
\end{corollary}

\begin{theorem}
In an intra-regular $\Gamma $-AG$^{\ast \ast }$-groupoid $S,$ $f\cap
g=f\circ _{\Gamma }g$ for every fuzzy $\Gamma $-right ideal $f$ and every
fuzzy $\Gamma $-left ideal $g$ of $S.$
\end{theorem}

\begin{proof}
Assume that $S$ is intra-regular $\Gamma $-AG$^{\ast \ast }$-groupoid. Let $%
f $ and $g$ be any fuzzy $\Gamma $-right and fuzzy $\Gamma $-left ideal of $%
S $, then%
\begin{equation*}
f\circ _{\Gamma }g\subseteq f\circ _{\Gamma }S\subseteq f\text{ and }f\circ
_{\Gamma }g\subseteq S\circ _{\Gamma }g\subseteq g\text{ which implies that }%
f\circ _{\Gamma }g\subseteq f\cap g\text{.}
\end{equation*}%
Since $S$ is an intra-regular, so for each $a\in S$ there exist $x,y\in S$
and $\beta ,\xi ,\gamma \in \Gamma $ such that $a=(x\beta (a\xi a))\gamma y$%
. Now let $\psi \in \Gamma ,$ then by using $(3)$, $(1)$, $(5)$ and $(4),$
we have 
\begin{eqnarray*}
a &=&(x\beta (a\xi a))\gamma y=(a\beta (x\xi a))\gamma y=(y\beta (x\xi
a))\gamma a \\
&=&((u\psi v)\beta (x\xi a))\gamma a=((a\psi x)\beta (v\xi u))\gamma a\text{.%
}
\end{eqnarray*}%
Therefore, we have%
\begin{eqnarray*}
(f\circ _{\Gamma }g)(a) &=&\bigvee\limits_{a=((a\psi x)\beta (v\xi u))\gamma
a}\left\{ f((a\psi x)\beta (v\xi u))\wedge g(a)\right\} \\
&\geq &f((a\psi x)\beta (v\xi u))\wedge g(a) \\
&\geq &f(a)\wedge g(a)=(f\cap g)(a)\text{.}
\end{eqnarray*}
\end{proof}

\begin{corollary}
\label{fg}In an intra-regular $\Gamma $-AG$^{\ast \ast }$-groupoid $S,$ $%
f\cap g=f\circ _{\Gamma }g$ for every fuzzy $\Gamma $-right ideal $f$ and $g$
of $S.$
\end{corollary}

\begin{theorem}
\label{semi1}The set of fuzzy $\Gamma $-two-sided ideals of an intra-regular 
$\Gamma $-AG$^{\ast \ast }$-groupoid $S$ forms a semilattice structure with
identity $S$.
\end{theorem}

\begin{proof}
Let $\mathbb{I}_{\Gamma }$ be the set of fuzzy $\Gamma $-two-sided ideals of
an intra-regular $\Gamma $-AG$^{\ast \ast }$-groupoid $S$ and $f$, $g$ and $%
h\in \mathbb{I}_{\Gamma }$, then clearly $\mathbb{I}_{\Gamma }$ is closed
and by corollory \ref{id} and corollory \ref{fg}, we have $f=f\circ _{\Gamma
}f$ and $f\circ _{\Gamma }g=f\cap g$, where $f$ and $g$ are fuzzy $\Gamma $%
-two-sided ideals of $S$. Clearly $f\circ _{\Gamma }g=g\circ _{\Gamma }f$,
and now by using $(1)$, we get $(f\circ _{\Gamma }g)\circ _{\Gamma
}h=(h\circ _{\Gamma }g)\circ _{\Gamma }f=f\circ _{\Gamma }(g\circ _{\Gamma
}h)$. Also by using $(1)$ and lemma \ref{sf}, we\ have 
\begin{equation*}
f\circ _{\Gamma }S=(f\circ _{\Gamma }f)\circ _{\Gamma }S=(S\circ _{\Gamma
}f)\circ _{\Gamma }f=f\circ _{\Gamma }f=f.
\end{equation*}
\end{proof}

A fuzzy $\Gamma $-two-sided ideal $f$ of a $\Gamma $-AG-groupoid $S$ is said
to be a $\Gamma $-strongly irreducible if and only if for fuzzy $\Gamma $%
-two-sided ideals $g$ and $h$ of $S$, $g\cap h\subseteq f$ implies that $%
g\subseteq f$ or $h\subseteq f$.

The set of fuzzy $\Gamma $-two-sided ideals of a $\Gamma $-AG-groupoid $S$
is called a $\Gamma $-totally ordered under inclusion if for any fuzzy $%
\Gamma $-two-sided ideals $f$ and $g$ of $S$ either $f\subseteq g$ or $%
g\subseteq f$.

A fuzzy $\Gamma $-two-sided ideal $h$ of a $\Gamma $-AG-groupoid $S$ is
called a $\Gamma $-fuzzy prime ideal of $S$, if for any fuzzy $\Gamma $%
-two-sided $f$ and $g$ of $S$, $f\circ _{\Gamma }g\subseteq h$, implies that 
$f\subseteq h$ or $g\subseteq h$.

\begin{theorem}
In an intra-regular $\Gamma $-AG$^{\ast \ast }$-groupoid $S,$ a fuzzy $%
\Gamma $-two-sided ideal is $\Gamma $-strongly irreducible if and only if it
is $\Gamma $-fuzzy prime.
\end{theorem}

\begin{proof}
It follows from corollory \ref{fg}.
\end{proof}

\begin{theorem}
Every fuzzy $\Gamma $-two-sided ideal of an intra-regular $\Gamma $-AG$%
^{\ast \ast }$-groupoid $S$ is $\Gamma $-fuzzy prime if and only if the set
of fuzzy $\Gamma $-two-sided ideals of $S$ is $\Gamma $-totally ordered
under inclusion.
\end{theorem}

\begin{proof}
It follows from corollory \ref{fg}.
\end{proof}

\begin{theorem}
\label{inte}For a fuzzy\ subset $f$\ of an intra-regular $\Gamma $-AG$^{\ast
\ast }$-groupoid, the following statements are equivalent.
\end{theorem}

$(i)$ $f$\ is a fuzzy\ $\Gamma $-two-sided ideal of $S.$

$(ii)$ $f$\ is a fuzzy\ $\Gamma $-interior ideal of $S$.

\begin{proof}
$(i)\Rightarrow (ii):$ Let $f$ be any fuzzy $\Gamma $-two- sided ideal of $S$%
, then obviously $f$ is a fuzzy $\Gamma $-interior ideal of $S$.

$(ii)\Rightarrow (i):$ Let $f$ be any fuzzy $\Gamma $-interior ideal of $S$
and $a$, $b\in S$. Since $S$ is an intra-regular $\Gamma $-AG-groupoid, so
for each $a,b\in S$ there exist $x,y,u,v\in S$ and $\beta ,\xi ,\gamma
,\delta ,\psi ,\eta \in \Gamma $ such that $a=(x\beta (a\xi a))\gamma y$ and 
$b=(u\delta (b\psi b))\eta v$. Now let $\alpha ,\in \Gamma $, thus by using $%
(1)$, $(3)$ and $(2)$, we have%
\begin{eqnarray*}
f\left( a\alpha b\right) &=&f\left( \left( x\beta (a\xi a)\right) \gamma
y\right) \alpha b)=f\left( \left( b\gamma y\right) \alpha \left( x\beta
(a\xi a)\right) \right) \\
&=&f\left( \left( b\gamma y\right) \alpha \left( a\beta \left( x\xi a\right)
\right) \right) =f\left( \left( b\gamma a\right) \alpha \left( y\beta \left(
x\xi a\right) \right) \right) \geq f\left( a\right) \text{.}
\end{eqnarray*}

Also by using $(3)$, $(4)$ and $(2),$ we have%
\begin{eqnarray*}
f\left( a\alpha b\right) &=&f\left( a\alpha \left( \left( u\delta (b\psi
b)\right) \eta v\right) \right) =f\left( \left( u\delta (b\psi b)\right)
\alpha \left( a\eta v\right) \right) =f\left( \left( b\delta (u\psi
b)\right) \alpha \left( a\eta v\right) \right) \\
&=&f\left( \left( v\delta a\right) \alpha \left( (u\psi b)\eta b\right)
\right) =f\left( (u\psi b)\alpha \left( \left( v\delta a\right) \eta
b\right) \right) \geq f(b)\text{.}
\end{eqnarray*}

Hence $f$ is a fuzzy $\Gamma $-two- sided ideal of $S.$
\end{proof}

\begin{theorem}
\label{q2}A fuzzy subset $f$ of an intra-regular $\Gamma $-AG$^{\ast \ast }$%
-groupoid is fuzzy\ $\Gamma $-two-sided ideal if and only if it is a fuzzy $%
\Gamma $-quasi ideal.
\end{theorem}

\begin{proof}
Let $f$ be any fuzzy $\Gamma $-two-sided ideal of $S$, then obviously $f$ is
a fuzzy $\Gamma $-quasi ideal of $S$.

Conversely, assume that $f$ is a fuzzy $\Gamma $-quasi ideal of $S$, then by
using corollory \ref{id} and $(4)$, we have 
\begin{equation*}
f\circ _{\Gamma }S=(f\circ _{\Gamma }f)\circ _{\Gamma }(S\circ _{\Gamma
}S)=(S\circ _{\Gamma }S)\circ _{\Gamma }(f\circ _{\Gamma }f)=S\circ _{\Gamma
}f\text{.}
\end{equation*}

Therefore $f\circ _{\Gamma }S=(f\circ _{\Gamma }S)\cap (S\circ _{\Gamma
}f)\subseteq f$. Thus $f$ is a fuzzy $\Gamma $-right ideal of $S$ and by
lemma \ref{llb}, $f$ is a fuzzy $\Gamma $-left ideal of $S$.
\end{proof}

\begin{theorem}
\label{gener}For a fuzzy\ subset $f$\ of an intra-regular $\Gamma $-AG$%
^{\ast \ast }$-groupoid $S,$ the following$\ $conditions are equivalent.
\end{theorem}

$(i)$ $f$\ is a fuzzy\ $\Gamma $-bi-ideal of $S.$

$(ii)$ $f$\ is a fuzzy $\Gamma $-generalized\ bi-ideal of $S$.

\begin{proof}
$(i)\Longrightarrow (ii):$ Let $f$ be any fuzzy $\Gamma $-bi-ideal of $S$,
then obviously $f$ is a fuzzy $\Gamma $-generalized bi-ideal of $S$.

$(ii)\Rightarrow (i):$ Let $f$ be any fuzzy $\Gamma $-generalized bi-ideal
of $S$, and $a$, $b\in S.$ Since $S$ is an intra-regular $\Gamma $%
-AG-groupoid, so for each $a\in S$ there exist $x,y\in S$ and $\beta ,\gamma
,\xi \in \Gamma $ such that $a=(x\beta (a\xi a))\gamma y$. Now let $\alpha
,\delta ,\in \Gamma $, then by using $(5),$ $(4)$, $(2)$ and $(3)$, we have%
\begin{eqnarray*}
f(a\alpha b) &=&f(((x\beta (a\xi a))\gamma y)\alpha b)=f(((x\beta (a\xi
a))\gamma (u\delta v))\alpha b) \\
&=&f(((v\beta u)\gamma ((a\xi a)\delta x))\alpha b)=f(((a\xi a)\gamma
((v\beta u)\delta x))\alpha b) \\
&=&f(((x\gamma (v\beta u))\delta (a\xi a))\alpha b)=f((a\delta ((x\gamma
(v\beta u))\xi a))\alpha b)\geq f(a)\wedge f(b)\text{.}
\end{eqnarray*}

Therefore $f$ is a fuzzy bi-ideal of $S$.
\end{proof}

\begin{theorem}
\label{bii}For a fuzzy\ subset $f$\ of an intra-regular $\Gamma $-AG$^{\ast
\ast }$-groupoid $S$, the following$\ $conditions are equivalent.
\end{theorem}

$(i)$ $f$\ is a fuzzy\ $\Gamma $-two- sided ideal of $S.$

$(ii)$ $f$\ is a fuzzy\ $\Gamma $-bi-ideal of $S$.

\begin{proof}
$(i)\Longrightarrow (ii):$ Let $f$ be any fuzzy $\Gamma $-two- sided ideal
of $S$, then obviously $f$ is a fuzzy $\Gamma $-bi-ideal of $S$.

$(ii)\Rightarrow (i):$ Let $f$ be any fuzzy $\Gamma $-bi-ideal of $S$. Since 
$S$ is an intra-regular $\Gamma $-AG-groupoid, so for each $a,b\in S$ there
exist $x,y,u,v\in S$ and $\beta ,\xi ,\gamma ,\delta ,\psi ,\eta \in \Gamma $
such that $a=(x\beta (a\xi a))\gamma y$ and $b=(u\delta (b\psi b))\eta v$.
Now let $\alpha \in \Gamma ,$ then by using $(1)$, $(4)$, $(2)$ and $(3)$,
we have%
\begin{eqnarray*}
f(a\alpha b) &=&f(((x\beta (a\xi a))\gamma y)\alpha b)=f((b\gamma y)\alpha
(x\beta (a\xi a))) \\
&=&f(((a\xi a)\gamma x)\alpha (y\beta b))=f(((y\beta b)\gamma x)\alpha (a\xi
a)) \\
&=&f(\left( a\gamma a\right) \alpha (x\xi (y\beta b)))=f(\left( (x\xi
(y\beta b))\gamma a\right) \alpha a) \\
&=&f(((x\xi (y\beta b))\gamma ((x\beta (a\xi a))\gamma y))\alpha a) \\
&=&f(((x\beta (a\xi a))\gamma ((x\xi (y\beta b))\gamma y))\alpha a) \\
&=&f(((y\beta (x\xi (y\beta b)))\gamma ((a\xi a)\gamma x))\alpha a) \\
&=&f(((a\xi a)\gamma ((y\beta (x\xi (y\beta b)))\gamma x))\alpha a) \\
&=&f(((x\xi (y\beta (x\xi (y\beta b))))\gamma \left( a\gamma a\right)
)\alpha a) \\
&=&f((a\gamma \left( (x\xi (y\beta (x\xi (y\beta b))))\gamma a\right)
)\alpha a) \\
&\geq &f(a)\wedge f(a)=f(a)\text{.}
\end{eqnarray*}

Now by using $(3)$, $(4)$ and $(1),$ we have%
\begin{eqnarray*}
f(a\alpha b) &=&f(a\alpha (u\delta (b\psi b))\eta v)=f((u\delta (b\psi
b))\alpha (a\eta v))=f((v\delta a)\alpha ((b\psi b)\eta u)) \\
&=&f((b\psi b)\alpha ((v\delta a)\eta u))=f(\left( ((v\delta a)\eta u)\psi
b\right) \alpha b) \\
&=&f(\left( ((v\delta a)\eta u)\psi ((u\delta (b\psi b))\eta v)\right)
\alpha b) \\
&=&f(\left( (u\delta (b\psi b))\psi (((v\delta a)\eta u)\eta v)\right)
\alpha b) \\
&=&f(\left( (v\delta ((v\delta a)\eta u))\psi ((b\psi b)\eta u)\right)
\alpha b) \\
&=&f(\left( (b\psi b)\psi ((v\delta ((v\delta a)\eta u))\eta u)\right)
\alpha b) \\
&=&f(\left( (u\psi (v\delta ((v\delta a)\eta u)))\psi (b\eta b)\right)
\alpha b) \\
&=&f(\left( b\psi ((u\psi (v\delta ((v\delta a)\eta u)))\eta b)\right)
\alpha b) \\
&\geq &f\left( b\right) \wedge f\left( b\right) =f\left( b\right) .
\end{eqnarray*}
\end{proof}

\begin{theorem}
Let $f$ be a subset of an intra-regular $\Gamma $-AG$^{\ast \ast }$-groupoid 
$S,$ then the following conditions are equivalent.
\end{theorem}

$(i)$ $f$ is a fuzzy $\Gamma $-bi-ideal of $S$.

$(ii)$ $(f\circ _{\Gamma }S)\circ _{\Gamma }f=f$ and $f\circ _{\Gamma }f=f.$

\begin{proof}
$(i)\Longrightarrow (ii):$ Let $f$ be a fuzzy $\Gamma $-bi-ideal of an
intra-regular $\Gamma $-AG$^{\ast \ast }$-groupoid $S$. Let $a\in A$, then
there exists $x,y\in S$ and $\beta ,\gamma ,\xi \in \Gamma $ such that $%
a=(x\beta (a\xi a))\gamma y$. Now let $\alpha \in \Gamma $, then by using $%
(3),$ $(1),$ $(5),$ $(4)$ and $(2),$ we have%
\begin{eqnarray*}
a &=&(x\beta (a\xi a))\gamma y=(a\beta (x\xi a))\gamma y=(y\beta (x\xi
a))\gamma a=(y\beta (x\xi ((x\beta (a\xi a))\gamma y)))\gamma a \\
&=&((u\alpha v)\beta (x\xi ((a\beta (x\xi a))\gamma y)))\gamma a=((((a\beta
(x\xi a))\gamma y)\alpha x)\beta (v\gamma u))\gamma a \\
&=&(((x\gamma y)\alpha (a\beta (x\xi a)))\beta (v\gamma u))\gamma
a=((a\alpha ((x\gamma y)\beta (x\xi a)))\beta (v\gamma u))\gamma a \\
&=&((a\alpha ((x\gamma x)\beta (y\xi a)))\beta (v\gamma u))\gamma
a=(((v\alpha u)\beta ((x\gamma x)\beta (y\xi a)))\gamma a)\gamma a \\
&=&(((v\alpha u)\beta ((x\gamma x)\beta (y\xi ((x\beta (a\xi a))\gamma
(u\alpha v)))))\gamma a)\gamma a \\
&=&(((v\alpha u)\beta ((x\gamma x)\beta (y\xi ((v\beta u)\gamma ((a\xi
a)\alpha x)))))\gamma a)\gamma a \\
&=&(((v\alpha u)\beta ((x\gamma x)\beta (y\xi ((a\xi a)\gamma ((v\beta
u)\alpha x)))))\gamma a)\alpha a \\
&=&(((v\alpha u)\beta ((x\gamma x)\beta ((a\xi a)\xi (y\gamma ((v\beta
u)\alpha x)))))\gamma a)\gamma a \\
&=&(((v\alpha u)\beta ((a\xi a)\beta ((x\gamma x)\xi (y\gamma ((v\beta
u)\alpha x)))))\gamma a)\gamma a \\
&=&(((a\xi a)\beta ((v\alpha u)\beta ((x\gamma x)\xi (y\gamma ((v\beta
u)\alpha x)))))\gamma a)\gamma a \\
&=&(((((x\gamma x)\xi (y\gamma ((v\beta u)\alpha x)))\xi (v\alpha u))\beta
(a\beta a))\gamma a)\gamma a \\
&=&((a\beta ((((x\gamma x)\xi (y\gamma ((v\beta u)\alpha x)))\xi (v\alpha
u))\beta a))\gamma a)\gamma a=(p\gamma a)\gamma a
\end{eqnarray*}

where $p=a\beta ((((x\gamma x)\xi (y\gamma ((v\beta u)\alpha x)))\xi
(v\alpha u))\beta a).$ Therefore%
\begin{eqnarray*}
((f\circ _{\Gamma }S)\circ _{\Gamma }f)(a)
&=&\dbigvee\limits_{a=(pa)a}\left\{ (f\circ _{\Gamma }S)(pa)\wedge
f(a)\right\} \geq \dbigvee\limits_{pa=pa}\left\{ f(p)\circ _{\Gamma
}S(a)\right\} \wedge f(a) \\
&\geq &\left\{ f(a\beta ((((x\gamma x)\xi (y\gamma ((v\beta u)\alpha x)))\xi
(v\alpha u))\beta a))\wedge S(a)\right\} \wedge f(a) \\
&\geq &f(a)\wedge 1\wedge f(a)=f(a).
\end{eqnarray*}

Now by using $(3),$ $(1),$ $(5)$and $(4),$ we have%
\begin{eqnarray*}
a &=&(x\beta (a\xi a))\gamma y=(a\beta (x\xi a))\gamma y=(y\beta (x\xi
a))\gamma a \\
&=&(y\beta (x\xi ((x\beta (a\xi a))\gamma (u\alpha v))))\gamma a=(y\beta
(x\xi ((v\beta u)\gamma ((a\xi a)\alpha x))))\gamma a \\
&=&(y\beta (x\xi ((a\xi a)\gamma ((v\beta u)\alpha x))))\gamma a=(y\beta
((a\xi a)\xi (x\gamma ((v\beta u)\alpha x))))\gamma a \\
&=&((a\xi a)\beta (y\xi (x\gamma ((v\beta u)\alpha x))))\gamma a=(((x\gamma
((v\beta u)\alpha x))\xi y)\beta (a\xi a))\gamma a \\
&=&(a\beta ((x\gamma ((v\beta u)\alpha x))\xi y))\gamma a=(a\beta p)\gamma a
\end{eqnarray*}

where $p=((x\gamma ((v\beta u)\alpha x))\xi y).$ Therefore%
\begin{eqnarray*}
((f\circ _{\Gamma }S)\circ f)(a) &=&\dbigvee\limits_{a=(a\beta p)\gamma
a}\left\{ (f\circ _{\Gamma }S)(a\beta p)\wedge f(a)\right\} \\
&=&\dbigvee\limits_{a=(a\beta p)\gamma a}\left( \dbigvee\limits_{a\beta
p=a\beta p}f(a)\wedge S(p)\right) \wedge f(a) \\
&=&\dbigvee\limits_{a=(a\beta p)\gamma a}\left\{ f(a)\wedge 1\wedge
f(a)\right\} =\dbigvee\limits_{a=(a\beta p)\gamma a}f(a)\wedge f(a) \\
&\leq &\dbigvee\limits_{a=(a\beta p)\gamma a}f((a\beta ((x\gamma ((v\beta
u)\alpha x))\xi y))\gamma a)=f(a).
\end{eqnarray*}

Thus $(f\circ _{\Gamma }S)\circ _{\Gamma }f=f.$ As we have shown that 
\begin{equation*}
a=((a\beta ((((x\gamma x)\xi (y\gamma ((v\beta u)\alpha x)))\xi (v\alpha
u))\beta a))\gamma a)\gamma a.
\end{equation*}
Let $a=p\gamma a$ where $p=(a\beta ((((x\gamma x)\xi (y\gamma ((v\beta
u)\alpha x)))\xi (v\alpha u))\beta a))\gamma a.$ Therefore%
\begin{eqnarray*}
(f\circ _{\Gamma }f)(a) &=&\dbigvee\limits_{a=p\gamma a}\left\{ f((a\beta
((((x\gamma x)\xi (y\gamma ((v\beta u)\alpha x)))\xi (v\alpha u))\beta
a))\gamma a)\wedge f(a)\right\} \\
&\geq &f(a)\wedge f(a)\wedge f(a)=f(a).
\end{eqnarray*}

Now by using lemma \ref{fghj}, we get $f\circ _{\Gamma }f=f.$

$(ii)\Longrightarrow (i):$ Assume that $f$ is a fuzzy subset of an
intra-regular $\Gamma $-AG$^{\ast \ast }$-groupoid $S$ and let $\beta
,\gamma \in \Gamma $, then%
\begin{eqnarray*}
f((x\beta a)\gamma y) &=&((f\circ _{\Gamma }S)\circ _{\Gamma }f)((x\beta
a)\gamma y)=\dbigvee\limits_{(x\beta a)\gamma y=(x\beta a)\gamma y}\{(f\circ
_{\Gamma }S)(xa)\wedge f(y)\} \\
&\geq &\dbigvee\limits_{x\beta a=x\beta a}\left\{ f(x)\wedge S(a)\right\}
\wedge f(y)\geq f(x)\wedge 1\wedge f(y)=f(x)\wedge f(z).
\end{eqnarray*}

Now by using lemma \ref{fghj}, $f$ is fuzzy $\Gamma $-bi-ideal of $S.$
\end{proof}

\begin{theorem}
Let $f$ be a subset of an intra-regular $\Gamma $-AG$^{\ast \ast }$-groupoid 
$S,$ then the following conditions are equivalent.
\end{theorem}

$(i)$ $f$ is a fuzzy $\Gamma $-interior ideal of $S$.

$(ii)$ $(S\circ _{\Gamma }f)\circ _{\Gamma }S=f.$

\begin{proof}
$(i)\Longrightarrow (ii):$ Let $f$ be a fuzzy $\Gamma $-bi-ideal of an
intra-regular $\Gamma $-AG$^{\ast \ast }$-groupoid $S$. Let $a\in A$, then
there exists $x,y\in S$ and $\beta ,\gamma ,\xi \in \Gamma $ such that $%
a=(x\beta (a\xi a))\gamma y$. Now let $\alpha \in \Gamma $, then by using $%
(3),$ $(5)$, $(4)$and $(1),$ we have%
\begin{eqnarray*}
a &=&(x\beta (a\xi a))\gamma y=(a\beta (x\xi a))\gamma y=((u\alpha v)\beta
(x\xi a))\gamma y \\
&=&((a\alpha x)\beta (v\xi u))\gamma y=(((v\xi u)\alpha x)\beta a)\gamma y.
\end{eqnarray*}

Therefore%
\begin{eqnarray*}
((S\circ _{\Gamma }f)\circ _{\Gamma }S)(a) &=&\dbigvee\limits_{a=(((v\xi
u)\alpha x)\beta a)\gamma y}\left\{ (S\circ _{\Gamma }f)(((v\xi u)\alpha
x)\beta a)\wedge S(y)\right\} \\
&\geq &\dbigvee\limits_{((v\xi u)\alpha x)\beta a=((v\xi u)\alpha x)\beta
a}\left\{ (S((v\xi u)\alpha x)\wedge f(a)\right\} \wedge 1 \\
&\geq &1\wedge f(a)\wedge 1=f(a).
\end{eqnarray*}

Now again%
\begin{eqnarray*}
((S\circ _{\Gamma }f)\circ _{\Gamma }S)(a) &=&\dbigvee\limits_{a=(x\beta
(a\xi a))\gamma y}\left\{ (S\circ _{\Gamma }f)(x\beta (a\xi a))\wedge
S(y)\right\} \\
&=&\dbigvee\limits_{a=(x\beta (a\xi a))\gamma y}\left(
\dbigvee\limits_{(x\beta (a\xi a))=(x\beta (a\xi a))}S(x)\wedge f(a\xi
a)\right) \wedge S(y) \\
&=&\dbigvee\limits_{a=(x\beta (a\xi a))\gamma y}\left\{ 1\wedge f(a\xi
a)\wedge 1\right\} =\dbigvee\limits_{a=(x\beta (a\xi a))\gamma y}f(a\xi a) \\
&\leq &\dbigvee\limits_{a=(x\beta (a\xi a))\gamma y}f((x\beta (a\xi
a))\gamma y)=f(a).
\end{eqnarray*}

Hence it follows that $(S\circ _{\Gamma }f)\circ _{\Gamma }S=f.$

$(ii)\Longrightarrow (i):$ Assume that $f$ is a fuzzy subset of an
intra-regular $\Gamma $-AG$^{\ast \ast }$-groupoid $S$ and let $\beta
,\gamma \in \Gamma $, then%
\begin{eqnarray*}
f((x\beta a)\gamma y) &=&((S\circ _{\Gamma }f)\circ _{\Gamma }S)((x\beta
a)\gamma y)=\dbigvee\limits_{(x\beta a)\gamma y=(x\beta a)\gamma y}\left\{
(S\circ _{\Gamma }f)(x\beta a)\wedge S(y)\right\} \\
&\geq &\dbigvee\limits_{x\beta a=x\beta a}\left\{ (S(x)\wedge f(a)\right\}
\wedge S(y)\geq f(a).
\end{eqnarray*}
\end{proof}

\begin{lemma}
\label{145}Let $f$ be a fuzzy subset of of an intra-regular $\Gamma $-AG$%
^{\ast \ast }$-groupoid $S,$ then $S\circ _{\Gamma }f=f=f\circ _{\Gamma }S$.
\end{lemma}

\begin{proof}
Let $f$ be a fuzzy\ $\Gamma $-left ideal of an intra-regular $\Gamma $-AG$%
^{\ast \ast }$-groupoid $S$ and let $a\in S,$ then there exists $x,y\in S$
and $\beta ,\gamma ,\xi \in \Gamma $ such that $a=(x\beta (a\xi a))\gamma y$.

Let $\alpha \in \Gamma $, then by using $(5),$ $(4)$ and $(3),$ we have%
\begin{eqnarray*}
a &=&(x\beta (a\xi a))\gamma y=(x\beta (a\xi a))\gamma (u\alpha v)=(v\beta
u)\gamma ((a\xi a)\alpha x)=(a\xi a)\gamma ((v\beta u)\alpha x) \\
&=&(x\xi (v\beta u))\gamma (a\alpha a)=a\gamma ((x\xi (v\beta u))\alpha a).
\end{eqnarray*}

Therefore%
\begin{eqnarray*}
(f\circ _{\Gamma }S)(a) &=&\dbigvee\limits_{a=a\gamma ((x\xi (v\beta
u))\alpha a)}\{f(a)\wedge S((x\xi (v\beta u))\alpha a)\} \\
&=&\dbigvee\limits_{a=a\gamma ((x\xi (v\beta u))\alpha a)}\{f(a)\wedge 1\} \\
&=&\dbigvee\limits_{a=a((x(ye))a)}f(a)=f(a).
\end{eqnarray*}

The rest of the proof can be followed from lemma \ref{sf}.
\end{proof}

\begin{theorem}
For a fuzzy\ subset $f$\ of an intra-regular $\Gamma $-AG$^{\ast \ast }$%
-groupoid $S,$ the following$\ $conditions are equivalent.
\end{theorem}

$(i)$ $f$\ is a fuzzy\ $\Gamma $-left ideal of $S.$

$(ii)$ $f$\ is a fuzzy\ $\Gamma $-right ideal of $S$.

$(iii)$ $f$\ is a fuzzy\ $\Gamma $-two-sided ideal of $S$.

$(iv)$ $f$\ is a fuzzy\ $\Gamma $-bi-ideal of $S$.

$(v)$ $f$\ is a fuzzy\ $\Gamma $-generalized bi- ideal of $S$.

$(vi)$ $f$\ is a fuzzy\ $\Gamma $-interior ideal of $S$.

$(vii)$ $f$\ is a fuzzy\ $\Gamma $-quasi ideal of $S$.

$(viii)$ $S\circ _{\Gamma }f=f=f\circ _{\Gamma }S$.

\begin{proof}
$(i)\Longrightarrow (viii):$ It follows from lemma \ref{145}.

$(viii)\Longrightarrow (vii):$ It is obvious.

$(vii)\Longrightarrow (vi):$ Let $f$ be a fuzzy\ $\Gamma $-quasi ideal of an
intra-regular $\Gamma $-AG$^{\ast \ast }$-groupoid $S$ and let $a\in S,$
then there exists $b,c\in S$ and $\beta ,\gamma ,\xi \in \Gamma $ such that $%
a=(b\beta (a\xi a))\gamma c.$ Let $\delta ,\eta \in \Gamma $, then by using $%
(3),$ $(4)$ and $(1)$, we have%
\begin{eqnarray*}
(x\delta a)\eta y &=&(x\delta (b\beta (a\xi a))\gamma c)\eta y=((b\beta
(a\xi a))\delta (x\gamma c))\eta y \\
&=&((c\beta x)\delta ((a\xi a)\gamma b))\eta y=((a\xi a)\delta ((c\beta
x)\gamma b))\eta y \\
&=&(y\delta ((c\beta x)\gamma b))\eta (a\xi a)=a\eta ((y\delta ((c\beta
x)\gamma b))\xi a)
\end{eqnarray*}

and from above%
\begin{eqnarray*}
(x\delta a)\eta y &=&(y\delta ((c\beta x)\gamma b))\eta (a\xi a)=(a\delta
a)\eta (((c\beta x)\gamma b)\xi y) \\
&=&((((c\beta x)\gamma b)\xi y)\delta a)\eta a.
\end{eqnarray*}%
Now by using lemma \ref{145}, we have%
\begin{eqnarray*}
f((x\delta a)\eta y) &=&((f\circ _{\Gamma }S)\cap (S\circ _{\Gamma
}f))((x\delta a)\eta y) \\
&=&(f\circ _{\Gamma }S)((x\delta a)\eta y)\wedge (S\circ _{\Gamma
}f)((x\delta a)\eta y).
\end{eqnarray*}

Now%
\begin{equation*}
(f\circ _{\Gamma }S)((x\delta a)\eta y)=\dbigvee\limits_{(x\delta a)\eta
y=a\eta ((y\delta ((c\beta x)\gamma b))\xi a)}\left\{ f(a)\wedge S_{\delta
}((y\delta ((c\beta x)\gamma b))\xi a)\right\} \geq f(a)
\end{equation*}

and 
\begin{equation*}
\left( S\circ _{\Gamma }f\right) ((x\delta a)\eta
y)=\dbigvee\limits_{(x\delta a)\eta y=((((cx)b)y)a)a}\left\{ S((((c\beta
x)\gamma b)\xi y)\delta a)\wedge f(a)\right\} \geq f(a).
\end{equation*}

This implies that $f((x\delta a)\eta y)\geq f(a)$ and therefore $f$ is a
fuzzy $\Gamma $-interior ideal of $S.$

$(vi)\Longrightarrow (v):$ It follows from theorems \ref{inte}, \ref{bii}
and \ref{gener}.

$(v)\Longrightarrow (iv):$ It follows from theorem \ref{gener}.

$(iv)\Longrightarrow (iii):$ It follows from theorem \ref{bii}.

$(iii)\Longrightarrow (ii):$ It is obvious and $(ii)\Longrightarrow (i)$ can
be followed from lemma \ref{llb}.
\end{proof}


\begin{thebibliography}{99}
\bibitem{p.hol} P. Holgate, Groupoids satisfying a simple invertive law, The
Math. Stud., $1-4,$ $61$ $(1992),$ $101-106.$

\bibitem{j1} J. Je\v{z}ek and T. Kepka, A note on medial division groupoids,
Proc. Amer. Math. Soc., $2,$ $119$ $(1993),$ $423-426.$

\bibitem{j2} J. Je\v{z}ek and T. Kepka, The equational theory of paramedial
cancellation groupoids, Czech. Math. J., $50,$ $125$ $(2000),$ $25-34.$

\bibitem{Kazi} M. A. Kazim and M. Naseeruddin, On almost semigroups, The
Alig. Bull. Math., $2$ $(1972),$ $1-7.$

\bibitem{Nouman} M. N. A. Khan, Fuzzification in Abel-Grassmann's groupoids,
MS. Thesis, $2009$.

\bibitem{N. Kuroki} N. Kuroki, Fuzzy bi-ideals in semigroups, Comment. Math.
Univ. St. Pauli, $27$ $(1979)$, $17-21$.{}

\bibitem{Mordeson} J. N. Mordeson, D. S. Malik and N. Kuroki, \textit{Fuzzy
semigroups}, Springer-Verlag, Berlin Germany, $2003$.

\bibitem{Mus3} Q. Mushtaq and S. M. Yusuf, On LA-semigroups, The\ Alig.
Bull. Math., $8$ $(1978),$ $65-70.$

\bibitem{Mushtaq and Yusuf} Q. Mushtaq and S. M. Yusuf, On LA-semigroup
defined by a commutative inverse semigroup, Math. Bech., $40$ $(1988)$, $%
59-62$.

\bibitem{biblo} Q. Mushtaq and M. Khan, M-systems in LA-semigroups, Souteast
Asian Bulletin of Mathematics, $33$ $(2009),$ $321-327$.

\bibitem{madis} Q. Mushtaq and M. Khan, Semilattice decomposition of locally
associative AG$^{\ast \ast }$-groupoids, Algeb. Colloq., $16,1$ $(2009),$ $%
17-22$.

\bibitem{nas} M. Naseeruddin, Some studies in almost semigroups and flocks,
Ph.D., thesis, Aligarh Muslim University, Aligarh, India, $1970.$

\bibitem{A. Rosenfeld} A. Rosenfeld, Fuzzy groups, J. Math. Anal. Appl., $35$
$(1971)$, $512-517$.

\bibitem{gemma1} M. K. Sen, On $\Gamma $-semigroups, Proceeding of
International Symposium on Algebra and Its Applications, Decker Publication,
New York, $(1981),301-308$.

\bibitem{in} T. Shah and I. Rehman, On M-systems in $\Gamma $-AG-groupoids,
Proc. Pakistan Acad. Sci., $47(1):33-39.$ $2010$.

\bibitem{Protic} N. Stevanovi\'{c} and P. V. Proti\'{c}, Composition of
Abel-Grassmann's 3-bands, Novi Sad, J. Math., $2$, $34$ $(2004)$, $175-182$.

\bibitem{L.A.Zadeh} L. A. Zadeh, Fuzzy sets, Inform. Control, $8$ $(1965)$, $%
338-353$.
\end{thebibliography}
\end{document}